\documentclass{article}
\usepackage[utf8]{inputenc}
\usepackage{amsmath} 
\usepackage{amsmath}
\usepackage{amssymb} 
\usepackage{amsfonts} 
 \usepackage{amsthm}
 \usepackage{mathtools}
 \usepackage{latexsym}
  \usepackage{authblk}
  \usepackage{amsmath} 
\usepackage{amsmath}
\usepackage{amssymb} 
\usepackage{amsfonts} 
 \usepackage{amsthm}
 \usepackage{mathtools}
 \usepackage{latexsym}
 \usepackage{amscd}
 \usepackage[all]{xy}
 \usepackage{mathrsfs}
 \usepackage{yhmath}
  
\newtheorem{dfn}{Definition}[section]
 \newtheorem{them}[dfn]{Theorem}
 \newtheorem{lem}[dfn]{Lemma}
 \newtheorem{prp}[dfn]{Proposition}

\newtheorem{cor}[dfn]{Corollary}
 \newtheorem{cla}[dfn]{Claim}
 \newtheorem{Que}[dfn]{Question}

\title{Existence of a positive hyperbolic Reeb orbit in three spheres with finite free group actions}

\author{Taisuke SHIBATA\thanks{Research Institute for Mathematical Sciences, Kyoto University, Kyoto 606-8502,
JAPAN. E-mail address: shibata@kurims.kyoto-u.ac.jp}}

\date{\today}
\begin{document}

\maketitle
%
%
%
%
%
%
%
%
%



\begin{abstract}
Let $(Y,\lambda)$ be a non-degenerate contact three manifold.  D. Cristfaro-Gardiner, M. Hutchings and  D. Pomerleano showed that if $c_{1}(\xi=\mathrm{Ker}\lambda)$ is torsion, then the Reeb vector field of $(Y,\lambda)$ has infinity many Reeb orbits otherwise $(Y,\lambda)$ is a lens space or three sphere with exaxtly two simple elliptic orbits. In the same paper, they also showed that if $b_{1}(Y)>0$, $(Y,\lambda)$ has a simple positive hyperbolic orbit directly from the isomorhphism between Seiberg-Witten Floer homology and Embedded contact homology. In addition to this, they asked whether $(Y,\lambda)$ with infinity many simple orbits also has a positive hyperbolic orbit under $b_{1}(Y)=0$. In the present paper, we answer this question under $Y \simeq S^{3}$ with nontrivial finite free group actions. In particular, it gives a positive answer in the case of a lens space $(L(p,q),\lambda)$ with odd $p$.
\end{abstract}



\section{Introduction}

Let $(Y,\lambda)$ be a contact three manifold and $X_{\lambda}$ be the Reeb vector field of it. That is, $X_{\lambda}$ is the unique vector field satisfying $i_{X_{\lambda}}\lambda=1$ and $d\lambda(X_{\lambda},\,\,)=0$.  A smooth map $\gamma : \mathbb{R}/T\mathbb{Z} \to Y$ is called a Reeb orbit with periodic $T$ if $\dot{\gamma} =  X_{\lambda}(\gamma)$ and simple if $\gamma$ is  a embedding map. In this paper, two Reeb orbits are considered equivalent if they differ by reparametrization.

The three-dimensional Weinstein conjecture which states that every contact closed three manifold $(Y,\lambda)$ has at least one simple periodic orbit  was shown by C. H. Taubes by using  Seiberg-Witten Floer (co)homology \cite{T1}.  After that,   D. Cristofaro-Gardiner and M. Hutchings showed that every contact closed three manifold $(Y,\lambda)$ has at least two simple periodic orbit by using Embedded contact homology(ECH) in \cite{CH}.  ECH was introduced by M. Hutchings in several papers (for example, it is briefly explained in \cite{H2}).

Let $\gamma : \mathbb{R}/T\mathbb{Z} \to Y$ be  a Reeb orbit with periodic $T$. If its return map $d\phi^{T}|_{\mathrm{Ker}\lambda=\xi} :\xi_{\gamma(0)} \to \xi_{\gamma(0)}$ has no eigenvalue 1, we call it a non-degenerate Reeb orbit and we call a contact manifold $(Y,\lambda)$ non-degenerate if all Reeb orbits are non-degenerate.

According to the eigenvalues of their return maps, non-degenerate periodic orbits are classified into three types. A periodic orbit  is negative hyperbolic if $d\phi^{T}|_{\xi}$ has eigenvalues $h,h^{-1} < 0$, positive hyperbolic if $d\phi^{T}|_{\xi}$ has eigenvalues $h,h^{-1} > 0$ and  elliptic if $d\phi^{T}|_{\xi}$ has eigenvalues $e^{\pm i2\pi\theta}$ for some $\theta \in \mathbb{R}\backslash \mathbb{Q}$.

For a non-degenerate contact three manifold $(Y,\lambda)$, the following theorems were proved using ECH in an essential way.

\begin{them}[\cite{HT3}]\label{exacttwo}
 Let $(Y,\lambda)$ be a closed contact non-degenerate three manifold. Assume that there exists exactly two simple Reeb orbit, then both of them are elliptic and $Y$ is a lens space (possibly $S^{3}$).
\end{them}
\begin{them}[\cite{HCP}]
Let $(Y,\lambda)$ be a non-degenerate contact three manifold. Let $\mathrm{Ker}\lambda=\xi$. Then
\item[1.] if $c_{1}(\xi)$ is torsion, there exists “infinitely many periodic orbits, or there exists exactly two elliptic simple periodic orbits and $Y$ is diffeomorphic to a lens space(that is, it reduces to \cite{HT3}).
\item [2.]if $c_{1}(\xi)$ is not torsion, there exists at least four periodic orbit.
\end{them}

\begin{them}[\cite{HCP}]\label{posihyp}
If $b_{1}(Y)>0$, there exists at least one positive hyperbolic orbit.
\end{them}
In general, ECH splits into two parts $\mathrm{ECH}_{\mathrm{even}}$ and $\mathrm{ECH}_{\mathrm{odd}}$. In particular, $\mathrm{ECH}_{\mathrm{odd}}$ is the part which detects the existence of a positive hyperbolic orbit and moreover if $b_{1}(Y)>0$, we can see directly from the isomorhphism between Seiberg-Witten Floer homology and ECH (Theorem \ref{test}) that $\mathrm{ECH}_{\mathrm{odd}}$ does not vanish. Theorem \ref{posihyp} was proved by using these facts.

In contrast to the case $b_{1}(Y)>0$, if $b_{1}(Y)=0$,  $\mathrm{ECH}_{\mathrm{odd}}$ may vanish, so this method doesn't work. As a generalization of this phenomena,  D. Cristfaro-Gardiner, M. Hutchings and  D. Pomerleano asked the next question in the same paper.

\begin{Que}[\cite{HCP}]\label{quest}
Let $Y$ be a closed connected three-manifold which is not $S^3$ or a lens space, and let $\lambda$ be a nondegenerate contact form on $Y$. Does $\lambda$ have a positive hyperbolic simple Reeb orbit?
\end{Que}

The reason why the cases $S^3$ and lens spaces are excluded in Question \ref{quest} is that they admit contact forms with exactly two simple elliptic orbits as stated in Theorem \ref{exacttwo} (for example, see \cite{HT3}).  So in general, we can change the assumption of Question \ref{quest} to the one that  $(Y,\lambda)$ is not a lens space or $S^3$ with exactly two elliptic orbits (this is generic condition. see \cite{Ir}).  

For this purpose, the author proved the next theorem in \cite{S}.

\begin{them}[\cite{S}]\label{elliptic}
Let $(Y,\lambda)$ be a nondegenerate contact three manifold with $b_{1}(Y)=0$. Suppose that $(Y,\lambda)$  has infinity many simple periodic orbits (that is, $(Y,\lambda)$ is not a lens space with exactly two simple Reeb orbits) and has at least one  elliptic orbit. Then, there exists at least one simple positive hyperbolic orbit.
\end{them}

By Theorem \ref{elliptic}, for answering Question \ref{quest}, we can see that it is enough to consider the next problem.

\begin{Que}
Let $Y$ be a closed connected three manifold with $b_{1}(Y)=0$. Does  $Y$  admit a non-degenerate contact form $\lambda$ such that all simple orbits are negative hyperbolic?
\end{Que}

The next theorems is the main theorem of this paper.

\begin{them}\label{maintheorem}
Let $L(p,q)$ $(p\neq \pm 1)$ be a lens space with odd $p$. Then $L(p,q)$ can not admit a non-degenerate contact form $\lambda$  all of whose simple periodic orbits are negative hyperbolic.
\end{them}

Immediately, we have the next corollary.
\begin{cor}\label{periodic}
Let $(L(p,q),\lambda)$ $(p\neq \pm 1)$ be a lens space with a non-degenerate contact form $\lambda$. Suppose that $p$ is odd and there are infinity many simple Reeb orbits. Then, $(L(p,q),\lambda)$ has a simple positive hyperbolic orbit.
\end{cor}

In addition to Theorem \ref{maintheorem}, the next Theorem \ref{z2act} and Corollary \ref{allact} hold. The proof of Theorem \ref{z2act} is shorter than the one of Theorem \ref{maintheorem} and Corollary \ref{allact} follows from Theorem \ref{maintheorem} and Theorem \ref{z2act} immediately.  We prove them at the end of this paper.

\begin{them}\label{z2act}
Let $(S^{3},\lambda)$ be a non-degenerate contact three sphere with a free $\mathbb{Z}/2\mathbb{Z}$ action. Suppose that $(S^{3},\lambda)$ has infinitely many simple periodic orbits. Then $(S^{3},\lambda)$ has a simple positive hyperbolic orbit.
\end{them}

\begin{cor}\label{allact}
Let $(S^{3},\lambda)$ be a non-degenerate contact three sphere with a nontrivial  finite free group action. Suppose that $(S^{3},\lambda)$ has  infinitely many simple periodic orbits. Then $(S^{3},\lambda)$ has a simple positive hyperbolic orbit.
\end{cor}

\subsection*{Acknowledgement}
The author would like to thank his advisor Professor Kaoru Ono
for his discussion, and Suguru Ishikawa for a series of discussion. The author also would like to  thank Michael Hutching for some comments.
This work was supported by JSPS KAKENHI Grant Number JP21J20300.

\section{Preliminaries}
 
For a non-degenerate contact three manifold $(Y,\lambda)$ and $\Gamma \in H_{1}(Y;\mathbb{Z})$, Embedded contact homology $\mathrm{ECH}(Y,\lambda,\Gamma)$ is defined. At first, we define the chain complex $(\mathrm{ECC}(Y,\lambda,\Gamma),\partial)$. In this paper, we consider ECH over $\mathbb{Z}/2\mathbb{Z}=\mathbb{F}$.
\begin{dfn} [{\cite[Definition 1.1]{H1}}]\label{qdef}
An orbit set $\alpha=\{(\alpha_{i},m_{i})\}$ is a finite pair of distinct simple periodic orbit $\alpha_{i}$ with positive integer $m_{i}$.
If $m_{i}=1$ whenever $\alpha_{i}$ is hyperboric orbit, then $\alpha=\{(\alpha_{i},m_{i})\}$ is called an admissible orbit set.
\end{dfn}
Set $[\alpha]=\sum m_{i}[\alpha_{i}] \in H_{1}(Y)$. For two orbit set $\alpha=\{(\alpha_{i},m_{i})\}$ and $\beta=\{(\beta_{j},n_{j})\}$ with $[\alpha]=[\beta]$, we define  $H_{2}(Y,\alpha,\beta)$ to be the set of relative homology classes of
2-chains $Z$ in $Y$ with $\partial Z =\sum_{i}m_{i} \alpha_{i}-\sum_{j}m_{j}\beta_{j}$. This is an affine space over $H_{2}(Y)$.
\begin{dfn}[{\cite[Definition2.2]{H1}}]\label{representative}
Let $Z \in H_{2}(Y;\alpha,\beta)$. A representative of $Z$ is an immersed oriented compact surface $S$ in $[0,1]\times Y$ such that:
\item[1.] $\partial S$ consists of positively oriented (resp. negatively oriented) covers of $\{1\}\times \alpha_{i}$ (resp. $\{0\}\times \beta_{j}$) whose total multiplicity is $m_{i}$ (resp. $n_{j}$).
\item[2.] $[\pi (S)]=Z$, where $\pi:[0,1]\times Y \to Y$ denotes the projection.
\item[3.] $S$ is embedded in $(0,1)\times Y$, and $S$ is transverse to $\{0,1\}\times Y$.
\end{dfn}

From now on, we fix a trivialization of $\xi$ defined over every simple orbit $\gamma$ and write it by $\tau$.

For a non-degenerate Reeb orbit $\gamma$, $\mu_{\tau}(\gamma)$ denotes its Conley-Zehnder index with respect to a trivialization $\tau$ in this paper. If $\gamma$ is hyperbolic (that is, not elliptic), then $\mu_{\tau}(\gamma^{p})=p\mu_{\tau}(\gamma)$ for all positive integer $p$ where $\gamma^{p}$ denotes the $p$ times covering orbit of $\gamma$ (for example, see \cite[Proposition 2.1]{H1}).

\begin{dfn}[{\cite[{\S}8.2]{H1}}]\label{intersection}
Let $\alpha_{1}$, $\beta_{1}$, $\alpha_{2}$ and $\beta_{2}$ be orbit sets with $[\alpha_{1}]=[\beta_{1}]$ and $[\alpha_{2}]=[\beta_{2}]$. For a fixed trivialization $\tau$, we can define
\begin{equation}
    Q_{\tau}:H_{2}(Y;\alpha_{1},\beta_{1}) \times H_{2}(Y;\alpha_{2},\beta_{2}) \to \mathbb{Z}.
\end{equation}
by $Q_{\tau}(Z_{1},Z_{2})=-l_{\tau}(S_{1},S_{2})+\#(S_{1}\cap{S_{2}})$ where  $S_{1}$, $S_{2}$ are  representatives of $Z_{1}$, $Z_{2}$  for $Z_{1}\in H_{2}(Y;\alpha_{1},\beta_{1})$, $Z_{2} \in H_{2}(Y;\alpha_{2},\beta_{2})$ respectively, $\#(S_{1}\cap{S_{2}})$ is their algebraic intersection number and $l_{\tau}$ is a kind of crossing number (see {\cite[{\S}8.3]{H1}} for details).
\end{dfn}

\begin{dfn}[{\cite[Definition 1.5]{H1}}]
For $Z\in H_{2}(Y,\alpha,\beta)$, we define ECH index by
\begin{equation}
    I(\alpha,\beta,Z):=c_{1}(\xi|_{Z},\tau)+Q_{\tau}(Z)+\sum_{i}\sum_{k=1}^{m_{i}}\mu_{\tau}(\alpha_{i}^{k})-\sum_{j}\sum_{k=1}^{n_{j}}\mu_{\tau}(\beta_{j}^{k}).
\end{equation}
Here,   $c_{1}(\xi|_{Z},\tau)$ is a relative Chern number  and, $Q_{\tau}(Z)=Q_{\tau}(Z,Z)$. Moreover this is independent of $\tau$ (see  \cite{H1} for more details).
\end{dfn}
\begin{prp}[{\cite[Proposition 1.6]{H1}}]\label{indexbasicprop}
 The ECH index $I$ has the following properties.
  \item[1.] For orbit sets $\alpha, \beta, \gamma$ with $[\alpha]=[\beta]=[\gamma]=\Gamma\in H_{1}(Y)$ and $Z\in H_{2}(Y,\alpha,\beta)$, $Z'\in H_{2}(Y,\beta,\gamma)$,
  \begin{equation}\label{adtiv}
  I(\alpha,\beta,Z)+I(\beta,\gamma,Z')=I(\alpha,\gamma,Z+Z').
  \end{equation}
  \item[2.] For $Z, Z'\in H_{2}(Y,\alpha,\beta)$,
  \begin{equation}\label{homimi}
      I(\alpha,\beta,Z)-I(\alpha,\beta,Z')=<c_{1}(\xi)+2\mathrm{PD}(\Gamma),Z-Z'>.
  \end{equation}
  \item[3.] If $\alpha$ and $\beta$ are admissible orbit sets,
  \begin{equation}\label{mod2}
      I(\alpha,\beta,Z)=\epsilon(\alpha)-\epsilon(\beta) \,\,\,\mathrm{mod}\,\,2.
  \end{equation}
  Here, $\epsilon(\alpha)$, $\epsilon(\beta)$ are the numbers of positive hyperbolic orbits in $\alpha$, $\beta$ respectively.

\end{prp}

For $\Gamma \in H_{1}(Y)$, we define $\mathrm{ECC}(Y,\lambda,\Gamma)$ as freely generated module over $\mathbb{Z}/2$ by admissible orbit sets $\alpha$ such that $[\alpha]=\Gamma$. That is,
\begin{equation}
    \mathrm{ECC}(Y,\lambda,\Gamma):= \bigoplus_{\alpha:\mathrm{admissibe\,\,orbit\,\,set\,\,with\,\,}{[\alpha]=\Gamma}}\mathbb{Z}_{2}\langle \alpha \rangle.
\end{equation}

To define the differential $\partial:\mathrm{ECC}(Y,\lambda,\Gamma)\to \mathrm{ECC}(Y,\lambda,\Gamma) $, we pick a generic $\mathbb{R}$-invariant almost complex structure $J$  on $\mathbb{R}\times Y$ which satisfies $J(\frac{d}{ds})=X_{\lambda}$ and $J\xi=\xi$.

We consider $J$-holomorphic curves  $u:(\Sigma,j)\to (\mathbb{R}\times Y,J)$ where
the domain $(\Sigma, j)$ is a punctured compact Riemann surface. Here the domain $\Sigma$ is
not necessarily connected.  Let $\gamma$ be a (not necessarily simple) Reeb orbit.  If a puncture
of $u$ is asymptotic to $\mathbb{R}\times \gamma$ as $s\to \infty$, we call it a positive end of $u$ at $\gamma$ and if a puncture of $u$ is asymptotic to $\mathbb{R}\times \gamma$ as $s\to -\infty$, we call it a negative end of $u$ at $\gamma$ (For more details \cite{H1}).

Let $u:(\Sigma,j)\to (\mathbb{R}\times Y,J)$ and $u':(\Sigma',j')\to (\mathbb{R}\times Y,J)$ be two $J$-holomorphic curves. If there is a biholomorphic map $\phi:(\Sigma,j)\to (\Sigma',j')$ with $u'\circ \phi= u$, we regard $u$ and $u'$ as equivalent.

 Let $\alpha=\{(\alpha_{i},m_{i})\}$ and $\beta=\{(\beta_{i},n_{i})\}$ be orbit sets. Let $\mathcal{M}^{J}(\alpha,\beta)$ denote the set of  $J$-holomorphic curves with positive ends
at covers of $\alpha_{i}$ with total covering multiplicity $m_{i}$, negative ends at covers of $\beta_{j}$
with total covering multiplicity $n_{j}$, and no other punctures. Moreover, in $\mathcal{M}^{J}(\alpha,\beta)$, we consider two
$J$-holomorphic curves   to be equivalent if they represent the same current in $\mathbb{R}\times Y$.
For $u \in \mathcal{M}^{J}(\alpha,\beta)$, we naturally have $[u]\in H_{2}(Y;\alpha,\beta)$ and we set $I(u)=I(\alpha,\beta,[u])$. Moreover we define
\begin{equation}\label{just}
     \mathcal{M}_{k}^{J}(\alpha,\beta):=\{\,u\in  \mathcal{M}^{J}(\alpha,\beta)\,|\,I(u)=k\,\,\}
\end{equation}

In this notations, we can define $\partial_{J}:\mathrm{ECC}(Y,\lambda,\Gamma)\to \mathrm{ECC}(Y,\lambda,\Gamma)$ as follows.

For admissible orbit set $\alpha$ with $[\alpha]=\Gamma$, we define

\begin{equation}
    \partial_{J} \langle \alpha \rangle=\sum_{\beta:\mathrm{admissible\,\,orbit\,\,set\,\,with\,\,}[\beta]=\Gamma} \# (\mathcal{M}_{1}^{J}(\alpha,\beta)/\mathbb{R})\cdot \langle \beta \rangle.
\end{equation}

Note that  the above counting $\#$ is well-defined since $\mathcal{M}_{1}^{J}(\alpha,\beta)/\mathbb{R}$ is a set of finite points (this follows from Proposition \ref{ind} and the compactness of the moduli space (\cite{H1})). Moreover $\partial_{J} \circ \partial_{J}=0$ (see \cite{HT1}, \cite{HT2}) and the homology defined by $\partial_{J}$ does not depend on $J$ (see Theorem \ref{test}, or see \cite{T1}).

For $u\in \mathcal{M}^{J}(\alpha,\beta)$, the  its (Fredholm) index is defined by

\begin{equation}
    \mathrm{ind}(u):=-\chi(u)+2c_{1}(\xi|_{[u]},\tau)+\sum_{k}\mu_{\tau}(\gamma_{k}^{+})-\sum_{l}\mu_{\tau}(\gamma_{l}^{-}).
\end{equation}
Here $\{\gamma_{k}^{+}\}$ is the set consisting of (not necessarilly simple) all positive ends of $u$ and $\{\gamma_{l}^{-}\}$ is that one of all negative ends.  Note that for generic $J$, if $u$ is connected and somewhere injective, then the moduli space of $J$-holomorphic
curves near $u$ is a manifold of dimension $\mathrm{ind}(u)$ (see \cite[Definition 1.3]{HT1}).

 Let $\alpha=\{(\alpha_{i},m_{i})\}$ and $\beta=\{(\beta_{i},n_{i})\}$. For  $u\in \mathcal{M}^{J}(\alpha,\beta)$, it can be uniquely
written as $u=u_{0}\cup{u_{1}}$ where $u_{0}$ are unions of all components which maps to $\mathbb{R}$-invariant cylinders in $u$ and $u_{1}$ is the rest of $u$.

\begin{prp}[{\cite[Proposition 7.15]{HT1}}]\label{ind}
Suppose that $J$ is generic and $u=u_{0}\cup{u_{1}}\in \mathcal{M}^{J}(\alpha,\beta)$. Then

    \item[1.] $I(u)\geq 0$
    \item[2.] If $I(u)=0$, then $u_{1}=\emptyset$
    \item[3.] If $I(u)=1$, then $u_{1}$ is embedded and $u_{0}\cap{u_{1}}=\emptyset$. Moreover $\mathrm{ind}(u_{1})=1$.
    \item[4.] If $I(u)=2$ and $\alpha$ and $\beta$ are admissible, then $u_{1}$ is embedded and $u_{0}\cap{u_{1}}=\emptyset$. Moreover $\mathrm{ind}(u_{1})=2$.

\end{prp}

If $c_{1}(\xi)+2\mathrm{PD}(\Gamma)$ is torsion, there exists the relative $\mathbb{Z}$-grading.
 \begin{equation}
    \mathrm{ECH}(Y,\lambda,\Gamma):= \bigoplus_{*:\,\,\mathbb{Z}\mathrm{-grading}}\mathrm{ECH}_{*}(Y,\lambda,\Gamma).
\end{equation}

Let $Y$ be connected.
Then there is degree$-2$ map $U$.
\begin{equation}\label{Umap}
    U:\mathrm{ECH}_{*}(Y,\lambda,\Gamma) \to \mathrm{ECH}_{*-2}(Y,\lambda,\Gamma).
\end{equation}

To define this, choose a base point $z\in Y$ which is not on the image of any Reeb orbit and let $J$ be generic.
Then define a map 
\begin{equation}
     U_{J,z}:\mathrm{ECC}_{*}(Y,\lambda,\Gamma) \to \mathrm{ECC}_{*-2}(Y,\lambda,\Gamma)
\end{equation}
by
\begin{equation}
    U_{J,z} \langle \alpha \rangle=\sum_{\beta:\mathrm{admissible\,\,orbit\,\,set\,\,with\,\,}[\beta]=\Gamma} \# \{\,u\in \mathcal{M}_{2}^{J}(\alpha,\beta)/\mathbb{R})\,|\,(0,z)\in u\,\}\cdot \langle \beta \rangle.
\end{equation}

The above map $U_{J,z}$ is a chain map, and we define the $U$ map
as the induced map on homology.  Under the assumption, this map is independent of $z$ (for a generic $J$). See \cite[{\S}2.5]{HT3} for more details. Moreover, in the same reason as $\partial$, $U_{J,z}$ does not depend on $J$ (see Theorem \ref{test}, and see \cite{T1}). In this paper, we choose a suitable generic $J$ as necessary.

The next isomorphism is important.

\begin{them}[\cite{T1}]\label{test}
For each $\Gamma\in H_{1}(Y)$, there is an isomorphism
\begin{equation}
\mathrm{ECH}_{*}(Y,\lambda,\Gamma) \cong \check{HM}_{*}(-Y,\mathfrak{s}(\xi)+\mathrm{PD}(\Gamma))
\end{equation}
of relatively $\mathbb{Z}/d\mathbb{Z}$-graded abelian groups. Here $d$ is the divisibility of $c_{1}(\xi)+2\mathrm{PD}(\Gamma)$ in $H_{1}(Y)$ mod torsion and $\mathfrak{s}(\xi)$ is the
spin-c structure associated to the oriented 2–plane field  as in \cite{KM}.
Moerover, the above isomorphism interchanges
the map $U$ in (\ref{Umap}) with the map
\begin{equation}
   U_{\dag}: \check{HM}_{*}(-Y,\mathfrak{s}(\xi)+\mathrm{PD}(\Gamma)) \longrightarrow \check{HM}_{*-2}(-Y,\mathfrak{s}(\xi)+\mathrm{PD}(\Gamma))
\end{equation}
defined in \cite{KM}.

\end{them}

Here $\check{HM}_{*}(-Y,\mathfrak{s}(\xi)+2\mathrm{PD}(\Gamma))$ is a version of Seiberg-Witten Floer homology with
$\mathbb{Z}/2\mathbb{Z}$ coefficients defined by Kronheimer-Mrowka \cite{KM}.

The action of an orbit set $\alpha=\{(\alpha_{i},m_{i})\}$ is defined by 
\begin{equation}
A(\alpha)=\sum m_{i}A(\alpha_{i})=\sum m_{i}\int_{\alpha_{i}}\lambda. 
\end{equation}

Note that if two admissible orbit sets $\alpha=\{(\alpha_{i},m_{i})\}$ and $\beta=\{(\beta_{i},n_{i})\}$ have $A(\alpha)\leq A(\beta)$, then the coefficient of $\beta$ in $\partial \alpha$  is $0$ because of the positivity of $J$ holomorphic curves over $d\lambda$ and the fact that $A(\alpha)-A(\beta)$ is equivalent to the integral value of $d\lambda$ over $J$-holomorphic punctured curves which is asymptotic to $\alpha$ at $+\infty$, $\beta$ at $-\infty$.

Suppose that $b_{1}(Y)=0$. In this situation,  for any orbit sets $\alpha$ and $\beta$ with $[\alpha]=[\beta]$, $H_{2}(Y,\alpha,\beta)$ consists of only one component since $H_{2}(Y)=0$. So we may omit the homology component from the notation of ECH index $I$, that is, $I(\alpha,\beta)$ just denotes the  ECH index. Furthermore, for a orbit set $\alpha$ with $[\alpha]=0$, we set $I(\alpha):=I(\alpha,\emptyset)$.

\begin{prp}\label{liftinglinear}
Let $(Y,\lambda)$ be a non-degenerate connected contact three manifold with $b_{1}(Y)=0$. Let $\rho: \Tilde{Y}\to Y$ be a $p$-fold cover with $b_{1}(\Tilde{Y})=0$. Let $(\Tilde{Y},\Tilde{\lambda})$ be  a non-degenerate contact three manifold induced by the covering map. Suppose that $\alpha$ and $\beta$ be admissible orbit sets in $(Y,\lambda)$ consisting of only hyperbolic orbits. 
Then
\begin{equation}
    I(\rho^{*}\alpha,\rho^{*}\beta)=pI(\alpha,\beta)
\end{equation}
Here $\rho^{*}\alpha$ and $\rho^{*}\beta$ are inverse images of $\alpha$, $\beta$ and so admissible orbit sets in $(\Tilde{Y},\Tilde{\lambda})$.
\end{prp}
\begin{proof}[\bf Proof of Proposition \ref{liftinglinear}]
Let $\tau$ be a fixed trivialization of $\xi=\mathrm{Ker}\lambda$ defined over every simple orbit in $(Y,\lambda)$ and $\Tilde{\tau}$ be its induced trivialization of $\Tilde{\xi}=\mathrm{Ker}\Tilde{\lambda}$ defined over every simple orbit in $(\Tilde{Y},\Tilde{\lambda})$. More precisely, it is defined as follows. Let $\Tilde{\gamma}:\mathbb{R}/T\mathbb{Z}\to \Tilde{Y}$ be a  simple orbits. Since $\rho: \Tilde{Y}\to Y$ is a $p$-fold cover, there are a positive integer $r$ and simple orbit  $\gamma:\mathbb{R}/\frac{T}{r}\mathbb{Z}\to Y$ with $\gamma\circ \pi=\rho \circ \Tilde{\gamma}$ where $\pi:\mathbb{R}/T\mathbb{Z}\to \mathbb{R}/\frac{T}{r}\mathbb{Z}$ is the natural projection and $p$ is divisible by $r$. Let $\pi_{*}:\Tilde{\gamma}^{*}\Tilde{\xi}\to \gamma^{*}\xi$ be the induced bundle map. Then for a trivialization $\tau:\gamma^{*}\xi\to \mathbb{R}/\frac{T}{r}\mathbb{Z}\times \mathbb{C}$, we define the induced trivialization $\Tilde{\tau}:\Tilde{\gamma}^{*}\Tilde{\xi}\to \mathbb{R}/T\mathbb{Z}\times \mathbb{C}$ so that $(\pi\times \mathrm{id}) \circ \Tilde{\tau}=\tau \circ \pi_{*}$.

Let $\{Z\}=H_{2}(Y;\alpha,\beta)$ and $\{\Tilde{Z}\}=H_{2}(\Tilde{Y};\rho^{*}\alpha,\rho^{*}\beta)$. It is sufficient to prove that each term in $I(\rho^{*}\alpha,\rho^{*}\beta)$ defined by $\Tilde{\tau}$ and $\Tilde{Z}$ is $p$ times of corresponding term in $I(\alpha,\beta)$ defined by $\tau$ and $Z$.

At first, we consider the term of Conley-Zehnder indices of orbits. As just before Definition \ref{intersection}, for every hyperbolic orbit $\gamma$ in $(Y,\lambda)$, we have $\mu_{\tau}(\gamma^p)=p\mu_{\tau}(\gamma)$ and so for a simple hyperbolic orbit $\gamma$ in $(Y,\lambda)$, we have $p\mu_{\tau}(\gamma)=\mu_{\Tilde{\tau}}(\rho^{*}\gamma)$ (in the right hand side, if several orbits appear in $\rho^{*}\gamma$, we add their Conley-Zehnder indices all together). 

Next, consider the terms $c_{1}(\xi|_{Z},\tau)$. Recall that this is defined as follows. Let $S$ be a surface with boundary and $f:S\to Y$ be a map representing $Z$. Take a generic section $\psi:S\to f^{*}\xi$ such that $\psi|_{\partial S}$ is a nonvanishing $\tau$-trivial section and $\psi$ is  transverse to the zero section (see \cite{H1} for the  definition of $\tau$-trivial).  then $c_{1}(\xi|_{Z},\tau)$ is defined as the signed number of zeroes of $\psi$.

Take $f$, $S$ and $\psi$ as above. Since $\rho: \Tilde{Y}\to Y$ is a $p$-fold cover, we can take  a surface $\Tilde{S}$, a map $\Tilde{f}:\Tilde{S}\to \Tilde{Y}$,  and a $p$-fold cover $i:\Tilde{S} \to S$ so that $f\circ i= \rho \circ \Tilde{f}$. Since $i$ induces a bundle map $i_{*}:\Tilde{f}^{*}\Tilde{\xi}\to f^{*}\xi$, we have an induced section $i^{*}\psi:\Tilde{S}\to \Tilde{f}^{*}\Tilde{\xi}$. We can check directly that $\Tilde{f}:\Tilde{S}\to \Tilde{Y}$ represents $\Tilde{Z}$, $i^{*}\psi|_{\partial \Tilde{S}}$ is a nonvanishing $\Tilde{\tau}$-trivial section and $i^{*}\psi$ is  transverse to the zero section. This implies that $c_{1}(\Tilde{\xi}|_{\Tilde{Z}},\Tilde{\tau})$ is the signed number of zeroes of $i^{*}\psi$ and so by construction we have $c_{1}(\Tilde{\xi}|_{\Tilde{Z}},\Tilde{\tau})=pc_{1}(\xi|_{Z},\tau)$.

Finally, we consider $Q_{\tau}(Z)$.  Let  $f_{1}:S_{1}\to Y$, $f_{2}:S_{2}\to Y$ be two surfaces used to define $Q_{\tau}(Z)$ (see Definition \ref{intersection}). In the same way as  $c_{1}(\xi|_{Z},\tau)$, we can take $p$-fold covering $i_{1}:\Tilde{S_{1}} \to S_{1}$ and $i_{2}:\Tilde{S_{2}} \to S_{2}$ so that $f_{1}\circ i_{1}= \rho \circ \Tilde{f_{1}}$ and $f_{2}\circ i_{2}= \rho \circ \Tilde{f_{2}}$. We can use them to define $Q_{\Tilde{\tau}}(\Tilde{Z})$.  By construction, we have $\#(\Tilde{S_{1}}\cap{\Tilde{S_{2}}})=p\#(S_{1}\cap{S_{2}})$ and $l_{\Tilde{\tau}}(\Tilde{S_{1}},\Tilde{S_{2}})=pl_{\tau}(S_{1},S_{2})$. This implies that $Q_{\Tilde{\tau}}(\Tilde{Z})=pQ_{\tau}(Z)$. 

This completes the proof of Proposition \ref{liftinglinear}. 
\end{proof}

\section{Proof of the results}

There are isomorphisms as follows. 

\begin{prp}\label{isomorhpisms}
\begin{equation}\label{isomors3}
    \mathrm{ECH}(S^{3},\lambda,0)=\mathbb{F}[U^{-1},U]/U\mathbb{F}[U]
\end{equation}
and for any $\Gamma\in H_{1}(L(p,q))$,
\begin{equation}\label{isomorlll3}
    \mathrm{ECH}(L(p,q),\lambda,\Gamma)=\mathbb{F}[U^{-1},U]/U\mathbb{F}[U].
\end{equation}
\end{prp}

\begin{proof}[\bf Proof of Proposition \ref{isomorhpisms}]

These come from the isomorphisms between ECH, Seiberg-Witten Floer homology and Heegaard Floer homology. See Theorem \ref{test}, \cite{KM}, \cite{T2}, \cite{OZ} and \cite{KLT}.

Note that without the isomorphisms to Seiberg-Witten Floer homology,  we can compute  ECH of $S^{3}$ and $L(p,q)$ themselves under ellipsoids (see \cite[Section 3.7]{H3}) and  M. Hutchings gives a sketch of the proof of (\ref{isomors3}) and(\ref{isomorlll3}) including the structure of $U$-map   in a special case (see \cite[Section 4.1]{H3}).

\end{proof}

\begin{lem}\label{meces3}
Suppose that all simple  orbits in $(S^{3}, \lambda)$ are negative hyperbolic.
Then, there is a sequence of admissible orbit sets $\{\alpha_{i}\}_{i=0,1,2,...}$ satisfying the following conditions.

\item[1.] For any admissible orbit set $\alpha$, $\alpha$ is in $\{\alpha_{i}\}_{i=0,1,2,...}$.

\item[2.] $A(\alpha_{i})<A(\alpha_{j})$ if and only if $i<j$.

\item[3.] $I(\alpha_{i},\alpha_{j})=2(i-j)$ for any $i,\,\,j$ .
\end{lem}
\begin{proof}[\bf Proof of Lemma \ref{meces3}]
The assumption that there is no simple positive hyperbolic orbit means $\partial= 0$ because of the fourth statement in Proposition \ref{indexbasicprop}. So the ECH is isomorphic to a free module generated by all admissible orbit sets over $\mathbb{F}$. 
Moreover, from (\ref{isomors3}), we can see that for every two admissible orbit sets $\alpha$ and $\beta$ with $A(\alpha)>A(\beta)$, $U^{k} \langle \alpha \rangle= \langle \beta \rangle$ for some $k>0$. So for every non negative even number $2i$, there is  exactly one admissible orbit set $\alpha_{i}$ whose ECH index relative to $\emptyset$ is equal to $2i$.  By considering these arguments, we obtain Lemma \ref{meces3}.
\end{proof}

In the same way as before, we also obtain the next Lemma.
\begin{lem}\label{mecelens}
Suppose that all simple  orbits in $(L(p,q), \lambda)$ are negative hyperbolic.
Then, for any $\Gamma\in H_{1}(L(p,q))$, there is a sequence of admissible orbit sets $\{\alpha_{i}^{\Gamma}\}_{i=0,1,2,...}$ satisfying the following conditions.

\item[1.] For any $i=0,1,2,...$, $[\alpha_{i}^{\Gamma}]=\Gamma$ in $H_{1}(L(p,q))$.

\item[2.] For any admissible orbit set $\alpha$ with $[\alpha]=\Gamma$, $\alpha$ is in $\{\alpha_{i}^{\Gamma}\}_{i=0,1,2,...}$.

\item[3.] $A(\alpha_{i}^{\Gamma})<A(\alpha_{j}^{\Gamma})$ if and only if $i<j$.

\item[4.] $I(\alpha_{i}^{\Gamma},\alpha_{j}^{\Gamma})=2(i-j)$ for any $i,\,\,j$ .
\end{lem}

\begin{lem}\label{noncontractible}
Suppose that all simple  orbits in $(L(p,q), \lambda)$ are negative hyperbolic. Then there is no contractible simple orbit.
\end{lem}

\begin{proof}[\bf Proof of Lemma \ref{noncontractible}]
Let $\rho:(S^{3},\Tilde{\lambda})\to (L(p,q), \lambda)$ be the covering map where $\Tilde{\lambda}$ is the induced contact form of $\lambda$ by $\rho$. Suppose that there is a contractible simple orbit $\gamma$ in $(L(p,q),\lambda)$. Then the inverse image of $\gamma$ by $\rho$ consists of $p$ simple negative hyperbolic orbits. By symmetry, they have the same ECH index relative to $\emptyset$. This contradicts the results in Lemma \ref{meces3}.
\end{proof}

Recall the covering map $\rho:(S^{3},\Tilde{\lambda})\to (L(p,q), \lambda)$.

 By Lemma \ref{noncontractible}, we can see that under the assumptions, if $p$ is prime, there is an one-to-one correspondence between periodic orbits
in $(S^{3},\Tilde{\lambda})$ and ones in $(L(p,q), \lambda)$. That is, for each orbit $\gamma $ in $(L(p,q), \lambda)$, there is an unique orbit $\Tilde{\gamma}$ in $(S^{3},\Tilde{\lambda})$ satisfying $\rho^{*}\gamma=\Tilde{\gamma}$ and for each orbit $\Tilde{\gamma} $ in $(S^{3},\Tilde{\lambda})$, there is an unique orbit $\gamma$ in $(L(p,q), \lambda)$ satisfying $\rho^{*}\gamma=\Tilde{\gamma}$. As above, we distinguish orbits in $(S^{3},\Tilde{\lambda})$ from ones in  $(L(p,q), \lambda)$ by adding tilde. Furthermore, we also do the same way in orbit sets. That is, for each orbit set $\alpha=\{(\alpha_{i},m_{i})\}$ over $(L(p,q), \lambda)$, we set $\Tilde{\alpha}=\{(\Tilde{\alpha}_{i},m_{i})\}$.

\begin{lem}\label{s3lift}
Under the assumptions and notations in Lemma \ref{mecelens}, there is a labelling $\{\Gamma_{0},\,\Gamma_{1},....,\,\Gamma_{p-1}\}=H_{1}(L(p,q))$ satisfying the following conditions.
\item[1.] $\Gamma_{0}=0$ in $H_{1}(L(p.q))$.
\item[2.] If $A(\Tilde{\alpha}_{i}^{\Gamma_{j}})<A(\Tilde{\alpha}_{i'}^{\Gamma_{j'}})$, then $i<i'$ or $j<j'$.
\item[3.] For any $i=0,1,2....$ and $\Gamma_{j} \in \{\Gamma_{0},\,\Gamma_{1},....,\,\Gamma_{p-1}\}$,  $\frac{1}{2}I(\Tilde{\alpha_{i}}^{\Gamma_{j}})=j$ in $\mathbb{Z}/p\mathbb{Z}$.

\end{lem}

\begin{proof}[\bf Proof of Lemma \ref{s3lift}]
By Proposition \ref{liftinglinear}, for each $\Gamma\in H_{1}(L(p,q))$, $\frac{1}{2}I(\Tilde{\alpha}^{\Gamma}_{i})$ in $\mathbb{Z}/p\mathbb{Z}$ is independent of $i$. Moreover, by Lemma \ref{noncontractible}, every admissible orbit set of $(S^{3},\Tilde{\lambda})$ comes from some of $(L(p,q), \lambda)$ and so in the notation of Lemma \ref{meces3} and Lemma \ref{mecelens}, the set $\{\Tilde{\alpha}_{i}^{\Gamma}\}_{i=0,1,...,\,\,\Gamma\in H_{1}(L(p,q))}$ is exactly equivalent to $\{ \alpha_{i}\}$. These arguments imply Lemma \ref{s3lift} (see the below diagram).
\begin{equation}\label{thediagram}
\xymatrix{
\langle \emptyset = \Tilde{\alpha}_{0}^{\Gamma_{0}} \rangle & \langle \Tilde{\alpha}_{0}^{\Gamma_{1}} \rangle\ar[l]_{U}&\langle \Tilde{\alpha}_{0}^{\Gamma_{2}} \rangle\ar[l]_{U} & \ar[l]_{U}&\ar@{.}[l]&\ar[l]_{U}\langle \Tilde{\alpha}_{0}^{\Gamma_{p-1}} \rangle \\ \langle \Tilde{\alpha}_{1}^{\Gamma_{0}} \rangle \ar[urrrrr]^{U} & \langle \Tilde{\alpha}_{1}^{\Gamma_{1}} \rangle\ar[l]&\langle \Tilde{\alpha}_{1}^{\Gamma_{2}} \rangle\ar[l]_{U} & \ar[l]_{U}&\ar@{.}[l]&\ar[l]_{U}\langle \Tilde{\alpha}_{1}^{\Gamma_{p-1}} \rangle \\\langle \Tilde{\alpha}_{2}^{\Gamma_{0}} \rangle \ar[urrrrr]^{U} & \langle \Tilde{\alpha}_{2}^{\Gamma_{1}} \rangle\ar[l]&\langle \Tilde{\alpha}_{2}^{\Gamma_{2}} \rangle\ar[l]_{U} & \ar[l]_{U}&\ar@{.}[l] }
\end{equation}
\end{proof}

\begin{lem}\label{isomorphismcyclic}
Suppose that all simple orbits in $(L(p,q), \lambda)$ are negative hyperbolic and $p$ is prime. For $\Gamma\in H_{1}(L(p.q))$, we set $f(\Gamma_{j}) \equiv \frac{1}{2}I(\Tilde{\alpha}_{i}^{\Gamma_{j}})=j \in \mathbb{Z}/p\mathbb{Z}$ for some $i\geq 0$. Then this map has to be isomorphism as cyclic groups. Here we note that by Lemma \ref{s3lift}, this map has to be well-defined and a bijective map from $H_{1}(L(p.q))$ to $\mathbb{Z}/p\mathbb{Z}$.  
\end{lem}
\begin{proof}[\bf Proof of Lemma \ref{isomorphismcyclic}]

Since under the assumption there are infinity many simple orbits and $|H_{1}(L(p,q))|<\infty$, we can pick up $p$ different simple periodic orbits $\{\gamma_{1},\,\gamma_{2},....,\gamma_{p}\}$ in $(L(p,q), \lambda)$ with $[\gamma_{1}]=[\gamma_{2}]=...=[\gamma_{p}]=\Gamma$ for some $\Gamma\in H_{1}(L(p,q))$. Since there is no contractible simple orbit (Lemma \ref{noncontractible}), $\Gamma \neq 0$ and so $f(\Gamma)\neq 0$. For $i=1,2,...,p$, let $\Tilde{\gamma}_{i}$ be the corresponding orbit in $(S^{3},\Tilde{\lambda})$ of $\gamma_{i}$ and $C_{\Tilde{\gamma_{i}}}$ be a representative of $Z_{\Tilde{\gamma}_{i}}$ where $\{Z_{\Tilde{\gamma}_{i}}\}=H_{2}(S^{3};\Tilde{\gamma}_{i},\emptyset)$ (see Definition \ref{representative}). 

\begin{cla}\label{claimonly}
Suppose that $1\leq i,j \leq p$ and $i \neq j$. Then
the intersection number $\#([0,1]\times \Tilde{\gamma}_{i}\cap{C_{\Tilde{\gamma}_{j}}})$ in $\mathbb{Z}/p\mathbb{Z}$ does not depend on the choice of $i,j$ where $\#([0,1]\times \Tilde{\gamma}_{i}\cap{C_{\Tilde{\gamma}_{j}}})$ is the algebraic intersection number in $[0,1]\times Y$ (see Definition \ref{intersection}).
\end{cla}
\begin{proof}[\bf Proof of Claim \ref{claimonly}]

By the definition, we have
\begin{equation}
    \frac{1}{2}I(\Tilde{\gamma}_{i}\cup{\Tilde{\gamma}_{j}},\Tilde{\gamma}_{i})=\frac{1}{2}I(\Tilde{\gamma}_{j})+\#([0,1]\times \Tilde{\gamma}_{i}\cap{C_{\Tilde{\gamma}_{j}}})
\end{equation}
So in $\mathbb{Z}/p\mathbb{Z}$,
\begin{equation}
    \begin{split}
        \#([0,1]\times \Tilde{\gamma}_{i}\cap{C_{\Tilde{\gamma}_{j}}})&=\frac{1}{2}I(\Tilde{\gamma}_{i}\cup{\Tilde{\gamma}_{j}},\Tilde{\gamma}_{i})-\frac{1}{2}I(\Tilde{\gamma}_{j})\\
        =&\frac{1}{2}I(\Tilde{\gamma}_{i}\cup{\Tilde{\gamma}_{j}})-\frac{1}{2}I(\Tilde{\gamma}_{i})-\frac{1}{2}I(\Tilde{\gamma}_{j})
        =  f(2\Gamma)-2f(\Gamma)
    \end{split}
\end{equation}
This implies that the value $ \#([0,1]\times \Tilde{\gamma}_{i}\cap{C_{\Tilde{\gamma}_{j}}})$ in  $\mathbb{Z}/p\mathbb{Z}$ depends only on $f(2\Gamma)$ and $f(\Gamma)$. This complete the proof of Claim \ref{claimonly}.
\end{proof}

Return to the proof of Lemma \ref{isomorphismcyclic}. we set $l:= \#([0,1]\times \Tilde{\gamma}_{i}\cap{C_{\Tilde{\gamma}_{j}}}) \in \mathbb{Z}/p\mathbb{Z}$ for $i \neq j$.

In the same way as  Claim \ref{claimonly}, for $1\leq n \leq p$, we have
\begin{equation}
       \frac{1}{2}I(\bigcup_{1\leq i \leq n }\Tilde{\gamma}_{i},\bigcup_{1\leq i \leq n-1}\Tilde{\gamma}_{i})=\frac{1}{2}I(\Tilde{\gamma}_{n})+\sum_{1\leq i \leq n-1} \#([0,1]\times \Tilde{\gamma}_{i}\cap{C_{\Tilde{\gamma}_{n}}})
\end{equation}
and so
\begin{equation}
    f(n\Gamma)-f((n-1)\Gamma)=f(\Gamma)+(n-1)l \,\,\,\,\,\, \mathrm{in}\,\,\mathbb{Z}/p\mathbb{Z}.
\end{equation}

Suppose that $l \neq 0$. Since $p$ is prime, there is $1\leq k \leq p$ such that $f(\Gamma)+(k-1)l=0$. This implies that  $f(k\Gamma)-f((k-1)\Gamma)=0$. But this contradicts the bijectivity of $f$. So $l=0$ and therefore $f(n\Gamma)=nf(\Gamma)$. Since $f(\Gamma)\neq 0$, we have that $f$ is isomorphism.
\end{proof}

\begin{lem}\label{minimalandsecond}
Suppose that all simple  orbits in $(L(p,q), \lambda)$ are negative hyperbolic.
Let $\gamma_{\mathrm{min}}$ and $\gamma_{\mathrm{sec}}$ be orbits with smallest and second smallest actions in $(L(p,q), \lambda)$ respectively. Then,
\begin{equation}\label{index2minsec}
    I(\Tilde{\gamma}_\mathrm{min})=I(\Tilde{\gamma}_{\mathrm{sec}},\Tilde{\gamma}_{\mathrm{min}})=2
\end{equation}
and moreover,
\begin{equation}\label{keyindex}
    6< I(\Tilde{\gamma}_{\mathrm{min}}\cup{\Tilde{\gamma}_{\mathrm{sec}}}) \leq 2p.
\end{equation}

\end{lem}

\begin{proof}[\bf Proof of Lemma \ref{minimalandsecond}]
Consider the diagram (\ref{thediagram}) and Lemma \ref{meces3}. As admissible orbit sets, $\Tilde{\gamma}_{\mathrm{min}}$ and  $\Tilde{\gamma}_{\mathrm{sec}}$ correspond to $\Tilde{\alpha}_{0}^{\Gamma_{1}}$ and $\Tilde{\alpha}_{0}^{\Gamma_{2}}$ respectively. This implies (\ref{index2minsec}).

Next, we show the inequality (\ref{keyindex}). See $\Tilde{\alpha}_{1}^{\Gamma_{0}}$ in the diagram (\ref{thediagram}). By the diagram, $I(\Tilde{\alpha}_{1}^{\Gamma_{0}})=2p$. Moreover, this comes from $\alpha_{1}^{\Gamma_{0}}$ with $[\alpha_{1}^{\Gamma_{0}}]=0\in H_{1}(L(p,q))$. By Lemma \ref{noncontractible}, $\alpha_{1}^{\Gamma_{0}}$ has to consist of at least two negative hyperbolic orbits. This implies that $A(\Tilde{\gamma}_{\mathrm{min}}\cup{\Tilde{\gamma}_{\mathrm{sec}}})\leq A(\Tilde{\alpha}_{1}^{\Gamma_{0}})$ and so by Lemma \ref{meces3}, $I(\Tilde{\gamma}_{\mathrm{min}}\cup{\Tilde{\gamma}_{\mathrm{sec}}})\leq I(\Tilde{\alpha}_{1}^{\Gamma_{0}})=2p$. 

Considering the above argument and Lemma \ref{meces3}, it is enough to show that $I(\Tilde{\gamma}_{\mathrm{min}}\cup{\Tilde{\gamma}_{\mathrm{sec}}}) \neq 6$. We prove this by contradiction.
Suppose that $I(\Tilde{\gamma}_{\mathrm{min}}\cup{\Tilde{\gamma}_{\mathrm{sec}}}) = 6$.  Since $\Tilde{\gamma}_{\mathrm{sec}}$ corresponds to $\Tilde{\alpha}_{0}^{\Gamma_{2}}$, we have $I(\Tilde{\gamma}_{\mathrm{sec}})=4$ and so $I(\Tilde{\gamma}_{\mathrm{min}}\cup{\Tilde{\gamma}_{\mathrm{sec}}}, \Tilde{\gamma}_{sec})=2$.  
To consider the $U$-map, fix a generic almost complex structure $J$ on $\mathbb{R}\times S^{3}$. Consider the $U$-map $U\langle \Tilde{\alpha}_{0}^{\Gamma_{1}}=\Tilde{\gamma}_{\mathrm{min}} \rangle=\langle \emptyset \rangle$. This implies that for each generic point $z\in S^{3}$, there is an embedded $J$-holomorphic curve $C_{z}\in \mathcal{M}^{J}(\Tilde{\gamma}_{\mathrm{min}},\emptyset)$ through $(0,z)\in \mathbb{R}\times S^{3}$. By using this $C_{z}$, we have
\begin{equation}\label{rhs}
    I(\Tilde{\gamma}_{\mathrm{min}}\cup{\Tilde{\gamma}_{\mathrm{sec}}}, \Tilde{\gamma}_{\mathrm{sec}})=I(\mathbb{R}\times \Tilde{\gamma}_{\mathrm{sec}}\cup{C_{z}}).
\end{equation}
Note that the right hand side of (\ref{rhs}) is the ECH index of holomorphic curves $\mathbb{R}\times \Tilde{\gamma}_{\mathrm{sec}}\cup{C_{z}}$ (see just before (\ref{just})).
Since $ I(\Tilde{\gamma}_{\mathrm{min}}\cup{\Tilde{\gamma}_{\mathrm{sec}}})=2$ and Proposition \ref{ind}, we have $\mathbb{R}\times \Tilde{\gamma}_{\mathrm{sec}}\cap{C_{z}}=\emptyset$. 

Consider a sequence of holomorphic curves $C_{z}$ as $z \to \Tilde{\gamma}_{\mathrm{sec}}$. By the compactness argument, this sequence has a convergent subsequence and this has a limiting holomorphic curve $C_{\infty}$ which may be splitting into more than one floor. But in this case, $C_{\infty}$ can not split because the action of the positive end of $C_{\infty}$ is smallest value $A(\Tilde{\gamma}_{\mathrm{min}})$. This implies that  $\mathbb{R}\times \Tilde{\gamma}_{\mathrm{sec}}\cap{C_{\infty}} \neq \emptyset$. Since $I(\Tilde{\gamma}_{\mathrm{min}}\cup{\Tilde{\gamma}_{\mathrm{sec}}}, \Tilde{\gamma}_{\mathrm{sec}})=I(\mathbb{R}\times \Tilde{\gamma}_{\mathrm{sec}}\cup{C_{\infty}}) =2$, $\mathbb{R}\times \Tilde{\gamma}_{\mathrm{sec}}\cap{C_{\infty}} \neq \emptyset$ contradicts  the fourth statement in Proposition \ref{ind}. Therefore, we have $I(\Tilde{\gamma}_{\mathrm{min}}\cup{\Tilde{\gamma}_{\mathrm{sec}}}) \neq 6$ and thus complete the proof of Lemma \ref{minimalandsecond}.
\end{proof}

\begin{proof}[\bf Proof of Theorem \ref{maintheorem}]

We may assume that $p$ is prime because the condition that all simple periodic orbits are negative hyperbolic does not change under taking odd-fold covering. By Lemma \ref{isomorphismcyclic} and (\ref{index2minsec}), we have $[\gamma_{\mathrm{sec}}]=2[\gamma_{\mathrm{min}}]$ and so $[\gamma_{\mathrm{min}}\cup{\gamma_{\mathrm{sec}}}]=3[\gamma_{\mathrm{min}}]$ in $H_{1}(L(p,q))$. Since $f$ is isomorphic, we have $\frac{1}{2}I(\Tilde{\gamma}_{\mathrm{min}}\cup{\Tilde{\gamma}_{\mathrm{sec}}})=3$ in $\mathbb{Z}/p\mathbb{Z}$. But this can not occur  in the range of (\ref{keyindex}). This is a contradiction and we complete the proof of Theorem \ref{maintheorem}.
\end{proof}

\begin{proof}[\bf Proof of Theorem \ref{z2act}]
We prove this by contradiction.
By Theorem \ref{elliptic}, we may assume that there is no elliptic orbit and so that all simple orbits are negative hyperbolic.
In the same as Lemma \ref{meces3}, there is  exactly one admissible orbit set $\alpha_{i}$ whose ECH index relative to $\emptyset$ is equal to $2i$.  If there is a non $\mathbb{Z}/2\mathbb{Z}$-invariant orbit $\gamma$, by symmetry there
are two orbit with the same ECH index relative to $\emptyset$. this is a contradiction. So we may assume that  all simple orbits are $\mathbb{Z}/2\mathbb{Z}$-invariant.

Let $(\mathbb{RP}^3,\lambda')$ be the non-degenerate contact three manifold obtained as the quotient space of $(S^{3},\lambda)$ and $\gamma$ be a $\mathbb{Z}/2\mathbb{Z}$-invariant periodic orbit. Then, this orbit corresponds to a double covering of a non-contractible orbit $\gamma'$ in $(\mathbb{RP}^3,\lambda')$.  This implies that the eigenvalues of the return map of $\gamma$ are square of the ones of $\gamma'$. This means that  the eigenvalues of the return map of $\gamma$ are both positive and so $\gamma$ is positive hyperbolic.  This is a contradiction and so we complete the proof of Theorem \ref{z2act}.
\end{proof}

\begin{proof}[\bf Proof of Corollary \ref{allact}]
Note that if $(L(p,q), \lambda)$ has a simple positive hyperbolic orbit, then its covering space $(S^{3},\Tilde{\lambda})$ also has a simple positive
hyperbolic orbit. Considering a non-trivial cyclic subgroup acting on the contact three sphere together with Corollary \ref{periodic} and Theorem \ref{z2act}, we complete the proof of Corollary \ref{allact}.
\end{proof}

\end{document}